\documentclass[11pt,a4paper]{article}

\usepackage{inputenc}
\usepackage{amsmath}
\usepackage{bm}
\usepackage{bbold}
\usepackage{amsthm}
\usepackage{enumerate}

\usepackage[top=1.0in, left=1.0in, right=1.0in, bottom=1.0in]{geometry}

\usepackage{tikz}\tikzset{x=1cm,y=1cm,z=1cm}

\usepackage{pgfplots}\pgfplotsset{compat=1.16}

\usepackage[hyphens]{url}

\usepackage{hyperref}
\usepackage{breakurl}

\title{Algebraic solutions of tropical optimization problems\thanks{Lobachevskii J. Math. 2015, 36(4), 363--374; https://doi.org/10.1134/S199508021504006X}}

\author{N. Krivulin\thanks{Faculty of Mathematics and Mechanics, Saint Petersburg State University, 28 Universitetsky Ave., St.~Petersburg, 198504, Russia, 
nkk@math.spbu.ru.}
\thanks{This work was supported in part by the Russian Foundation for Humanities, Grant No. 13-02-00338.}
}

\date{}

\newtheorem{theorem}{Theorem}
\newtheorem{lemma}[theorem]{Lemma}
\newtheorem{corollary}[theorem]{Corollary}

\theoremstyle{definition}

\begin{document}

\maketitle

\begin{abstract}
We consider multidimensional optimization problems, which are formulated and solved in terms of tropical mathematics. The problems are to minimize (maximize) a linear or nonlinear function defined on vectors over an idempotent semifield, and may have constraints in the form of linear equations and inequalities. The aim of the paper is twofold: first to give a broad overview of known tropical optimization problems and solution methods, including recent results; and second, to derive a direct, complete solution to a new constrained optimization problem as an illustration of the algebraic approach recently proposed to solve tropical optimization problems with nonlinear objective functions.
\\

\\

\end{abstract}

\section{Introduction}

Tropical (idempotent) mathematics, which is concerned with the theory and applications of semirings with idempotent addition, dates back to the early 1960's, when a few innovative works by Pandit \cite{Pandit1961Anew}, Cuninghame-Green \cite{Cuninghamegreen1962Describing}, Giffler \cite{Giffler1963Scheduling}, Hoffman \cite{Hoffman1963Onabstract}, Vorob{'}ev \cite{Vorobjev1963Theextremal} and Romanovski{\u\i} \cite{Romanovskii1964Asymptotic} made their appearance. Since that time the literature on the topic has increased rapidly with several monographs, including those by Carr{\'e} \cite{Carre1979Graphs}, Cuninghame-Green \cite{Cuninghamegreen1979Minimax}, U.~Zimmermann \cite{Zimmermann1981Linear}, Baccelli et al. \cite{Baccelli1993Synchronization}, Kolokoltsov and Maslov \cite{Kolokoltsov1997Idempotent}, Golan \cite{Golan2003Semirings}, Heidergott, Olsder and van der Woude \cite{Heidergott2006Maxplus}, Gondran and Minoux \cite{Gondran2008Graphs}, and Butkovi\v{c} \cite{Butkovic2010Maxlinear}; as well as with a great body of contributed papers.

Optimization problems that are set up and solved in the framework of tropical mathematics arose even in the early papers \cite{Cuninghamegreen1962Describing,Hoffman1963Onabstract}, and now form an important research and applied field. Tropical optimization problems are formulated to minimize (maximize) a linear or nonlinear function defined on vectors over an idempotent semifield, and may have constraints given by linear equations and inequalities. These problems find applications in job scheduling, location analysis, transportation networks, discrete event systems, and decision making. Some problems in a rather common setting can be solved directly in an exact form under fairly general assumptions. The solutions available for other problems take the form of iterative algorithms, which produce a solution, or indicate that no solution exists.

The aim of this paper is twofold: first, to give a broad overview of known tropical optimization problems and solution methods, including recent results; and second, to derive a direct, complete solution to a new optimization problem with inequality constraints as a clear illustration of an efficient algebraic solution technique based on the approach, which was developed and applied in \cite{Krivulin2015Multidimensional,Krivulin2014Complete,Krivulin2014Constrained}. Under this approach, we introduce an additional variable to represent the minimum value of the objective function, and then reduce the problem to the solution of an inequality with the new variable in the role of a parameter. Both the existence conditions for the solution of the inequality and the inequality constraints are exploited to evaluate the parameter, whereas all solutions to the inequality are taken as a complete solution to the problem

The paper is organized as follows. We start in Section~\ref{S-PDR} with a short, concise introduction to basic definitions and notation of idempotent algebra. Section~\ref{S-SLI} includes direct, complete solutions to linear inequalities, which provide the basis for the solving of a new optimization problem. In Section \ref{S-OKPS}, we give an overview of known optimization problems and solution methods. Furthermore, we outline recent results in Section~\ref{S-RR}. Finally, a new tropical optimization problem is formulated and completely solved in Section~\ref{S-NCOP}.

\section{Preliminary Definitions and Results}
\label{S-PDR}

We start with a summary of basic definitions and notation of tropical mathematics to provide an appropriate formal framework for the overview of optimization problems and the derivation of new results in subsequent sections. The summary mainly follows the results in \cite{Krivulin2006Solution,Krivulin2009Onsolution,Krivulin2012New,Krivulin2015Multidimensional}, which focus on direct solutions in a compact vector form. For more details in both introductory and advanced levels, see, e.g., \cite{Cuninghamegreen1979Minimax,Carre1979Graphs,Zimmermann1981Linear,Baccelli1993Synchronization,Kolokoltsov1997Idempotent,Golan2003Semirings,Heidergott2006Maxplus,Akian2007Maxplus,Litvinov2007Themaslov,Gondran2008Graphs,Butkovic2010Maxlinear}.

\subsection{Idempotent Semifield}

Consider a system $\langle\mathbb{X},\mathbb{0},\mathbb{1},\oplus,\otimes\rangle$, where $\mathbb{X}$ is a set, which contains two distinct elements, zero $\mathbb{0}$ and one $\mathbb{1}$, and is closed under two operations, addition $\oplus$ and multiplication, such that $\langle\mathbb{X},\mathbb{0},\oplus\rangle$ is an idempotent commutative monoid, $\langle\mathbb{X}\setminus\{\mathbb{0}\},\mathbb{1},\otimes\rangle$ is a commutative group, and multiplication distributes over addition. Since all nonzero elements have multiplicative inverses, the system is referred to as the idempotent semifield.

Addition is idempotent, which implies that $x\oplus x=x$ for all $x\in\mathbb{X}$. We assume the semifield to be linearly ordered by a total order that is consistent with the partial order defined by $x\leq y$ if and only if $x\oplus y=y$. It follows from the definition that addition and multiplication are isotone, and that the inequality $x\oplus y\leq z$ is equivalent to the inequalities $x\leq z$ and $y\leq z$.

Multiplication is invertible to provide every $x\ne\mathbb{0}$ with an element $x^{-1}$ such that $x\otimes x^{-1}=\mathbb{1}$. The integer powers are defined for any nonzero $x\in\mathbb{X}$ and integer $p>0$ as follows: $x^{0}=\mathbb{1}$, $x^{p}=x\otimes x^{p-1}$, $x^{-p}=(x^{-1})^{p}$, and $\mathbb{0}^{p}=\mathbb{0}$. Moreover, the semifield is considered radicable (algebraically closed) to extend the power notation to the rational exponents.

From now on, we omit multiplication sign for better readability. The relation symbols and the optimization objectives are thought of in the sense of the above mentioned linear order on $\mathbb{X}$.

Examples include the real semifields $\mathbb{R}_{\max,+}=\langle\mathbb{R}\cup\{-\infty\},-\infty,0,\max,+\rangle$, $\mathbb{R}_{\min,+}=\langle\mathbb{R}\cup\{+\infty\},+\infty,0,\min,+\rangle$, $\mathbb{R}_{\max,\times}=\langle\mathbb{R}_{+}\cup\{0\},0,1,\max,\times\rangle$, and $\mathbb{R}_{\min,\times}=\langle\mathbb{R}_{+}\cup\{+\infty\},+\infty,1,\min,\times\rangle$, where $\mathbb{R}_{+}=\{x\in\mathbb{R}|x>0\}$.

Specifically, the semifield $\mathbb{R}_{\max,+}$ takes $-\infty$ to be the zero and $0$ to be the one. For every $x\in\mathbb{R}$, the multiplicative inverse $x^{-1}$ corresponds to the opposite number $-x$ in the conventional algebra. The power $x^{y}$ is defined for all $x,y\in\mathbb{R}$ and corresponds to the arithmetic product $xy$. The linear order on $\mathbb{R}_{\max,+}$ coincides with the natural order defined on $\mathbb{R}$.

In the semifield $\mathbb{R}_{\min,\times}$, we have $\mathbb{0}=+\infty$, $\mathbb{1}=1$, $\oplus=\min$, and $\otimes=\times$. Multiplicative inversion and exponentiation accept the usual interpretation. The linear order is opposite to the natural order on $\mathbb{R}$.

\subsection{Matrix Algebra}

Consider matrices over $\mathbb{X}$ and denote the set of matrices of $m$ rows and $n$ columns by $\mathbb{X}^{m\times n}$. A matrix with all entries equal to $\mathbb{0}$ is the zero matrix denoted $\bm{0}$. A matrix is row- (column-) regular if it has no zero rows (columns). A matrix is regular if it is both row and column regular.

Addition and multiplication of conforming matrices, and multiplication by scalars follow the standard rules, where the operations $\oplus$ and $\otimes$ are taken instead of the ordinary addition and multiplication. These operations are isotone in each argument with respect to the matrix relation $\leq$, which is considered entry-wise. The transpose of a matrix $\bm{A}$ is indicated by $\bm{A}^{T}$.

For every nonzero matrix $\bm{A}=(a_{ij})\in\mathbb{X}^{m\times n}$, we define the multiplicative conjugate transpose to be a matrix $\bm{A}^{-}=(a_{ij}^{-})\in\mathbb{X}^{n\times m}$ that has the entries $a_{ij}^{-}=a_{ji}^{-1}$ if $a_{ji}\ne\mathbb{0}$, and $a_{ij}^{-}=\mathbb{0}$ otherwise.

Now, examine square matrices in $\mathbb{X}^{n\times n}$. A matrix that has all diagonal entries equal to $\mathbb{1}$ and all off-diagonal entries equal to $\mathbb{0}$ is the identity matrix denoted $\bm{I}$. For every matrix $\bm{A}$ and integer $p>0$, the nonnegative matrix powers are defined as $\bm{A}^{0}=\bm{I}$, $\bm{A}^{p}=\bm{A}^{p-1}\bm{A}$.

The trace of a matrix $\bm{A}$ is given by $\mathop\mathrm{tr}\bm{A}=a_{11}\oplus\cdots\oplus a_{nn}$.

A matrix that has only one row (column) is a row (column) vector. The set of column vectors with $n$ elements is denoted $\mathbb{X}^{n}$. A vector with all zero elements is the zero vector. A vector is regular if it has no zero elements.

Let $\bm{A}\in\mathbb{X}^{n\times n}$ is a row-regular matrix and $\bm{x}\in\mathbb{X}^{n}$ a regular vector. Then, the vector $\bm{A}\bm{x}$ is clearly a regular vector. Moreover, if the matrix $\bm{A}$ is column-regular, then the row vector $\bm{x}^{T}\bm{A}$ is regular as well.

For every nonzero column vector $\bm{x}=(x_{i})$, the multiplicative conjugate transpose is a row vector $\bm{x}^{-}$ that has the elements $x_{i}^{-}=x_{i}^{-1}$ if $x_{i}\ne\mathbb{0}$, and $x_{i}^{-}=\mathbb{0}$ otherwise. Multiplicative conjugate transposition has some useful properties, which are not difficult to verify. First, note that, for any regular vectors $\bm{x}$ and $\bm{y}$, the element-wise inequality $\bm{x}\leq\bm{y}$ implies $\bm{x}^{-}\geq\bm{y}^{-}$.

Furthermore, for any nonzero vector $\bm{x}$, the identity $\bm{x}^{-}\bm{x}=\mathbb{1}$ holds. If the vector $\bm{x}$ is regular, then the inequality $\bm{x}\bm{x}^{-}\geq\bm{I}$ is valid as well.

A scalar $\lambda\in\mathbb{X}$ is an eigenvalue of a matrix $\bm{A}\in\mathbb{X}^{n\times n}$ if there exists a nonzero vector $\bm{x}\in\mathbb{X}^{n}$ such that $\bm{A}\bm{x}=\lambda\bm{x}$. The eigenvalue of $\bm{A}$, which is maximal in the sense of the order defined on $\mathbb{X}$, is called the spectral radius and given directly by $\lambda
=\mathop\mathrm{tr}\bm{A}\oplus\cdots\oplus\mathop\mathrm{tr}\nolimits^{1/n}(\bm{A}^{n})$.

For any matrix $\bm{A}=(a_{ij})$ and vector $\bm{x}=(x_{i})$, we introduce the functions
$$
\|\bm{A}\|
=
\bigoplus_{ij}a_{ij},
\qquad
\|\bm{x}\|
=
\bigoplus_{i}x_{i},
$$
which play the role of tropical analogues of matrix and vector norms.

\section{Solutions to Linear Inequalities}
\label{S-SLI}

In this section, we present direct, complete solutions to linear inequalities, which are used later. Suppose that, given a matrix $\bm{A}\in\mathbb{X}^{m\times n}$ and a regular vector $\bm{d}\in\mathbb{X}^{m}$, we find all regular vectors $\bm{x}\in\mathbb{X}^{n}$ that satisfy the inequality
\begin{equation}
\bm{A}\bm{x}
\leq
\bm{d}.
\label{I-Axd}
\end{equation}

The next result offers a solution given as a consequence of the solution to the corresponding equation \cite{Krivulin2009Onsolution,Krivulin2009Methods}, and by independent proof \cite{Krivulin2013Direct}. 

\begin{lemma}
\label{L-Axd}
For any column-regular matrix $\bm{A}$ and regular vector $\bm{d}$, all regular solutions to inequality \eqref{I-Axd} are given by the inequality $\bm{x}\leq(\bm{d}^{-}\bm{A})^{-}$.
\end{lemma}

Furthermore, we consider the problem: given a matrix $\bm{A}\in\mathbb{X}^{n\times n}$ and a vector $\bm{b}\in\mathbb{X}^{n}$, find all regular vectors $\bm{x}\in\mathbb{X}^{n}$ to solve the inequality
\begin{equation}
\bm{A}\bm{x}\oplus\bm{b}
\leq
\bm{x}.
\label{I-Axbx}
\end{equation}

To describe a solution to the problem, we apply a function that assigns to every matrix $\bm{A}\in\mathbb{X}^{n\times n}$ a scalar given by $\mathop\mathrm{Tr}(\bm{A})=\mathop\mathrm{tr}\bm{A}\oplus\cdots\oplus\mathop\mathrm{tr}\bm{A}^{n}$.

Provided that $\mathop\mathrm{Tr}(\bm{A})\leq\mathbb{1}$, we use the asterate operator (also known as the Kleene star), which maps $\bm{A}$ to the matrix $\bm{A}^{\ast}=
\bm{I}\oplus\bm{A}\oplus\cdots\oplus\bm{A}^{n-1}$.

A direct, complete solution to inequality \eqref{I-Axbx} is obtained in \cite{Krivulin2006Solution,Krivulin2009Onsolution,Krivulin2015Multidimensional}.
\begin{theorem}\label{T-Axbx}
For any matrix $\bm{A}$ and vector $\bm{b}$, the following holds:
\begin{enumerate}
\item If $\mathop\mathrm{Tr}(\bm{A})\leq\mathbb{1}$, then all regular solutions to inequality \eqref{I-Axbx} are given by $\bm{x}=\bm{A}^{\ast}\bm{u}$, where $\bm{u}$ is any regular vector such that $\bm{u}\geq\bm{b}$.
\item If $\mathop\mathrm{Tr}(\bm{A})>\mathbb{1}$, then there is no regular solution.
\end{enumerate}
\end{theorem}

In the solution to an optimization problem below, Lemma~\ref{L-Axd} is used repeatedly to obtain intermediate results, whereas Theorem~\ref{T-Axbx} provides the basis for the solution, which reduces the problem to an inequality in the form of \eqref{I-Axbx}.

\section{Overview of Known Problems and Solutions}
\label{S-OKPS}

Since the early works by Cuninghame-Green \cite{Cuninghamegreen1962Describing} and Hoffman \cite{Hoffman1963Onabstract}, multidimensional optimization problems have become an important research domain in tropical mathematics. These problems appeared in different application contexts, including job scheduling \cite{Cuninghamegreen1962Describing,Cuninghamegreen1976Projections,Cuninghamegreen1979Minimax,Zimmermann1981Linear,Zimmermann1984Some,Zimmermann2006Interval,Butkovic2009Introduction,Butkovic2009Onsome,Tam2010Optimizing,Aminu2012Nonlinear}, location analysis \cite{Cuninghamegreen1991Minimax,Zimmermann1992Optimization,Cuninghamegreen1994Minimax,Krivulin2011Extremal,Krivulin2012New}, transportation networks \cite{Zimmermann1981Linear,Zimmermann2006Interval}, discrete event dynamic systems \cite{Gaubert1995Resource,Deschutter1996Maxalgebraic,Deschutter2001Model,Krivulin2005Evaluation} and decision making \cite{Elsner2004Maxalgebra,Elsner2010Maxalgebra,Gursoy2013Theanalytic}.

In this section, we offer a short overview of known tropical optimization problems and briefly discuss existing solution methods. The problems are to minimize or maximize linear and nonlinear functions defined on vectors over an idempotent semifield, subject to linear inequality and equality constraints. The overview covers the problems with those nonlinear functions which are, or can be, represented by means of multiplicative conjugate transposition of vectors. As our review of the literature shows, many problems, which are relevant to tropical optimization, have objective functions that admit this form. Note that these problems are usually considered in a different setting; they are formulated and solved either in the framework of ordinary mathematics \cite{Zimmermann1984Onmaxseparable,Zimmermann1984Some,Zimmermann1992Optimization,Zimmermann2003Disjunctive,Zimmermann2006Interval}, or in both terms of dual semifields such as $\mathbb{R}_{\max,+}$ and $\mathbb{R}_{\min,+}$ in \cite{Cuninghamegreen1976Projections,Cuninghamegreen1979Minimax,Butkovic2009Onsome,Tam2010Optimizing}.

To represent the problems below, we use the boldface capital letters $\bm{A}$, $\bm{B}$, and $\bm{C}$ for known matrices, the boldface lower-case letters $\bm{b}$, $\bm{d}$, $\bm{p}$, and $\bm{q}$ for known vectors, lower-case letters $r$ and $s$ for scalars. The symbol $\bm{x}$ stands for the unknown vector. The matrix and vector operations are meant in terms of an idempotent semifield. The minus sign in the superscript indicates multiplicative conjugate transposition. The relation symbols and problem objectives are in the sense of the order induced by idempotent addition.

\subsection{Problems with Linear Objective Functions}

One of the early problems in tropical optimization was a formal analogue of linear programming problems, which can be written in the form
\begin{equation*}
\begin{aligned}
&
\text{minimize}
&&
\bm{p}^{T}\bm{x},
%\\
&
\text{subject to}
&&
\bm{A}\bm{x}
\geq
\bm{d}.
\end{aligned}
\end{equation*}

Exact solutions to the problem have been obtained under various assumptions about the idempotent semiring, which provided the solution context. Specifically, Hoffman \cite{Hoffman1963Onabstract} considered a rather general idempotent semiring and proposed a solution based on an abstract extension of the duality principle in linear programming. Another general solution was derived by U.~Zimmermann \cite{Zimmermann1981Linear} by means of a residual-based solution technique. The common approach suggested in \cite{Hoffman1963Onabstract} was further developed by Superville \cite{Superville1978Various} to handle the problem in the context of the semifield $\mathbb{R}_{\max,+}$. The problem was also examined by Gavalec and K.~Zimmermann \cite{Gavalec2012Duality} within the framework of max-separable functions to obtain solutions for the semifields $\mathbb{R}_{\max,+}$ and $\mathbb{R}_{\max,\times}$.

Furthermore, K.~Zimmermann \cite{Zimmermann1984Onmaxseparable,Zimmermann1984Some,Zimmermann1992Optimization,Zimmermann2003Disjunctive,Zimmermann2006Interval} applied the results of the theory of max-separable functions to solve a problem with more constraints in the form 
\begin{equation*}
\begin{aligned}
&
\text{minimize}
&&
\bm{p}^{T}\bm{x},
%\\
&
\text{subject to}
&&
\bm{A}\bm{x}
\leq
\bm{d},
&&
\bm{C}\bm{x}
\geq
\bm{b},
&&
\bm{g}
\leq
\bm{x}
\leq
\bm{h}.
\end{aligned}
\end{equation*}

An exact solution to this problem has been obtained, which was, however, given in ordinary terms rather than in terms of tropical mathematics.

Butkovi\v{c} \cite{Butkovic1984Onproperties}, Butkovi\v{c} and Aminu \cite{Butkovic2009Introduction,Aminu2012Nonlinear} studied a problem that has a two-sided equality constraint (with the unknown vector on both sides) in the form
\begin{equation*}
\begin{aligned}
&
\text{minimize}
&&
\bm{p}^{T}\bm{x},
%\\
&
\text{subject to}
&&
\bm{A}\bm{x}\oplus\bm{b}
=
\bm{C}\bm{x}\oplus\bm{d}.
\end{aligned}
\end{equation*}

A pseudo-polynomial algorithm, which produces a solution if it exists or indicates that the problem has no solution, has been proposed in \cite{Butkovic2009Introduction}. The algorithm uses an alternating method \cite{Cuninghamegreen2003Theequation} to replace the two-sided equation by opposite inequalities, and then to alternately solve them to achieve more and more accurate estimates for the solution of the equation. Another heuristic approach, which combines a search scheme to find approximate solutions with iterative procedures to solve low-dimensional problems, was suggested in \cite{Aminu2012Nonlinear}.

\subsection{Problems with Nonlinear Objective Functions}

Optimization problems with nonlinear objective functions given by transposition and multiplicative conjugation of vectors form a rich class of problems arising in many applications. We divide these problems into groups according to the form of objective functions and to the principal interpretation of the problems.

\subsubsection{Chebyshev Approximation}

Among the first optimization problems with nonlinear objective functions was the problem formulated as
\begin{equation*}
\begin{aligned}
&
\text{minimize}
&&
(\bm{A}\bm{x})^{-}\bm{p},
%\\
&
\text{subject to}
&&
\bm{A}\bm{x}
\leq
\bm{p}.
\end{aligned}
\end{equation*}

This problem was examined by Cuninghame-Green \cite{Cuninghamegreen1976Projections} in the context of approximation in the semifield $\mathbb{R}_{\max,+}$ with the Chebyshev metric. The problem is to obtain vectors $\bm{x}$ that provide the best underestimating approximation to a vector $\bm{p}$ by means of vectors $\bm{A}\bm{x}$. A direct solution has been proposed based on the theory of linear operators on vectors over the semifield $\mathbb{R}_{\max,+}$. A similar solution was suggested by U.~Zimmermann \cite{Zimmermann1981Linear}.

An unconstrained approximation problem in terms of a general idempotent semifield was considered by Krivulin \cite{Krivulin2009Methods} in the form
\begin{equation}
\begin{aligned}
&
\text{minimize}
&&
(\bm{A}\bm{x})^{-}\bm{p}\oplus\bm{p}^{-}\bm{A}\bm{x}.
\end{aligned}
\label{P-AxppAx}
\end{equation}

An exact solution to the problem involves the derivation of a sharp lower bound for the objective function. The form of this bound is exploited to construct a vector at which the objective function attains the bound. The results obtained are then applied to solve the equation $\bm{A}\bm{x}=\bm{p}$ in a closed vector form. As a consequence, direct solutions were given for the following problems of underestimating and overestimating approximation:
\begin{equation*}
\begin{aligned}
&
\text{minimize}
&&
(\bm{A}\bm{x})^{-}\bm{p},
\\
&
\text{subject to}
&&
\bm{A}\bm{x}
\leq
\bm{p};
\end{aligned}
\qquad\qquad\qquad
\begin{aligned}
&
\text{minimize}
&&
\bm{p}^{-}\bm{A}\bm{x},
\\
&
\text{subject to}
&&
\bm{A}\bm{x}
\geq
\bm{p}.
\end{aligned}
%\label{P-AxppAxAxpAxp}
\end{equation*}

There are optimization problems that were originally formulated in a different framework, but can be readily represented in terms of tropical mathematics. Specifically, a constrained problem of Chebyshev approximation in the semifield $\mathbb{R}_{\max,+}$ was examined and solved in the ordinary setting with a polynomial-time threshold-type algorithm in \cite{Zimmermann1984Some}. This problem can be written as a tropical optimization problem in the form
\begin{equation}
\begin{aligned}
&
\text{minimize}
&&
(\bm{A}\bm{x})^{-}\bm{p}\oplus\bm{p}^{-}\bm{A}\bm{x},
%\\
&
\text{subject to}
&&
\bm{g}
\leq
\bm{x}
\leq
\bm{h}.
\end{aligned}
\label{P-AxppAxgxh}
\end{equation}

\subsubsection{Problems with Span Seminorm}

The problems, which were analyzed by Butkovi\v{c} and Tam \cite{Butkovic2009Onsome,Tam2010Optimizing} in the context of the semifield $\mathbb{R}_{\max,+}$, have the objective function in the form of the span seminorm defined as the maximum deviation between the elements of a vector. A solution technique was applied based on a combined formalism of both semifields $\mathbb{R}_{\max,+}$ and $\mathbb{R}_{\min,+}$. A representation of the problems in terms of $\mathbb{R}_{\max,+}$ alone is as follows
\begin{equation}
\begin{aligned}
&
\text{minimize}
&&
\bm{1}^{T}\bm{A}\bm{x}(\bm{A}\bm{x})^{-}\bm{1},
\end{aligned}
\qquad\qquad
\begin{aligned}
&
\text{maximize}
&&
\bm{1}^{T}\bm{A}\bm{x}(\bm{A}\bm{x})^{-}\bm{1},
\end{aligned}
\label{P-minmax1AxAx1}
\end{equation}
where $\bm{1}=(\mathbb{1},\ldots,\mathbb{1})^{T}$ is a vector of ones in the sense of $\mathbb{R}_{\max,+}$.

\subsubsection{A Problem of ``Linear-Fractional'' Programming}

A constrained minimization problem, which has a two-sided inequality constraint and is formulated in the semifield $\mathbb{R}_{\max,+}$ in the form
\begin{equation*}
\begin{aligned}
&
\text{minimize}
&&
(\bm{p}^{T}\bm{x}\oplus r)(\bm{q}^{T}\bm{x}\oplus s)^{-1},
%\\
&
\text{subject to}
&&
\bm{A}\bm{x}\oplus\bm{b}
\leq
\bm{C}\bm{x}\oplus\bm{d},
\end{aligned}
\end{equation*}
was investigated by Gaubert, Katz and Sergeev \cite{Gaubert2012Tropical} under the name of the tropical linear-fractional programming problem. The problem was solved using an iterative scheme based on the relationship \cite{Akian2012Tropical} between solutions of two-sided vector equations in the sense of $\mathbb{R}_{\max,+}$ and mean payoff games.

\subsubsection{Extremal Property of Spectral Radius}

The problem, examined in detail by Cuninghame-Green \cite{Cuninghamegreen1962Describing,Cuninghamegreen1979Minimax} in terms of the semifield $\mathbb{R}_{\max,+}$, was apparently the first optimization problem appeared in the context of tropical mathematics. With the use of conjugate transposition, the problem is given by
\begin{equation}
\begin{aligned}
&
\text{minimize}
&&
\bm{x}^{-}\bm{A}\bm{x}.
\end{aligned}
\label{P-minxAx}
\end{equation}

As one of the main results in \cite{Cuninghamegreen1962Describing}, it has been shown that the minimum value in the problem is equal to the spectral radius $\lambda$ of the matrix $\bm{A}$ and this value is attained at any eigenvector that satisfies the equality $\bm{A}\bm{x}=\lambda\bm{x}$. Moreover, explicit expressions have been derived to calculate the spectral radius and an eigenvector in terms of standard arithmetic operations. Similar results were obtained by Engel and Schneider \cite{Engel1975Diagonal} and by Superville \cite{Superville1978Various}.

In \cite{Cuninghamegreen1979Minimax} the above results were extended and described in general terms of tropical mathematics. To find all solutions, a computational approach was proposed consisted of solving a linear programming problem. Analogues solutions, which are represented in a compact vector form using multiplicative conjugate transposition, were derived by Krivulin \cite{Krivulin2005Evaluation,Krivulin2006Eigenvalues,Krivulin2009Methods}.

Finally, Elsner and van den Driessche \cite{Elsner2004Maxalgebra,Elsner2010Maxalgebra} have observed that not only the eigenvectors solve the problem, but all vectors $\bm{x}$ satisfying the inequality $\bm{A}\bm{x}\leq\lambda\bm{x}$ do that as well. An iterative computational procedure was suggested to find solutions to the inequality.

\section{Recent Results}
\label{S-RR}

In this section, we consider several new problems, which are formulated and solved in \cite{Krivulin2015Multidimensional,Krivulin2016Maximization,Krivulin2013Direct,Krivulin2013Explicit,Krivulin2015Extremal,Krivulin2014Constrained,Krivulin2014Complete} in terms of a general idempotent semifield. We offer direct, exact solutions which are represented in a compact vector form. For many problems, the results obtained provide complete solutions.

\subsection{Chebyshev-Like Approximation Problems}

We start with problems that have nonlinear objective functions similar to that in approximation problems \eqref{P-AxppAx} and \eqref{P-AxppAxgxh}. Applications of these problems include single facility minimax location problems in multidimensional spaces with Chebyshev distance under linear inequality and equality constraints (see, e.~g., \cite{Krivulin2012New,Krivulin2013Direct,Krivulin2014Complete}). 

To solve the next two problems with simple boundary constraints, we apply an approach which was developed in \cite{Krivulin2005Evaluation,Krivulin2006Eigenvalues,Krivulin2009Methods}. The solution is based on the derivation of a sharp lower bound on the objective function and the subsequent use of this bound to obtain all vectors that solve the problem.

First, we consider the following problem: given vectors $\bm{p},\bm{q},\bm{g},\bm{h}\in\mathbb{X}^{n}$, find regular vectors $\bm{x}\in\mathbb{X}^{n}$ that
\begin{equation}
\begin{aligned}
&
\text{minimize}
&&
\bm{q}^{-}\bm{x}\oplus\bm{x}^{-}\bm{p},
%\\
&
\text{subject to}
&&
\bm{g}
\leq
\bm{x}
\leq
\bm{h}.
\end{aligned}
\label{P-qxxpgxh}
\end{equation}

\begin{theorem}[\cite{Krivulin2013Direct}]
\label{T-qxxpgxh}
Let $\bm{p}$ and $\bm{q}$ be regular vectors, $\bm{g}$ and $\bm{h}$ be vectors such that $\bm{g}\leq\bm{h}$. Then, the minimum in \eqref{P-qxxpgxh} is equal to $\mu=(\bm{q}^{-}\bm{p})^{1/2}\oplus\bm{q}^{-}\bm{g}\oplus\bm{h}^{-}\bm{p}$, and all regular solutions are given by $\mu^{-1}\bm{p}\oplus\bm{g}\leq\bm{x}\leq(\mu^{-1}\bm{q}^{-}\oplus\bm{h}^{-})^{-}$.
\end{theorem}

Suppose that, given a matrix $\bm{A}\in\mathbb{X}^{m\times n}$ and vectors $\bm{p},\bm{q}\in\mathbb{X}^{m}$, $\bm{g}\in\mathbb{X}^{n}$, the problem is to obtain regular vectors $\bm{x}\in\mathbb{X}^{n}$ that
\begin{equation}
\begin{aligned}
&
\text{minimize}
&&
\bm{q}^{-}\bm{A}\bm{x}\oplus(\bm{A}\bm{x})^{-}\bm{p},
%\\
&
\text{subject to}
&&
\bm{x}
\geq
\bm{g}.
\end{aligned}
\label{P-qAxAxpxg}
\end{equation}

\begin{theorem}[\cite{Krivulin2013Direct}]
\label{T-qAxAxpxg}
Let $\bm{A}$ be a regular matrix, $\bm{p}$ and $\bm{q}$ be regular vectors. Then, the minimum in \eqref{P-qAxAxpxg} is equal to $\mu=((\bm{A}(\bm{q}^{-}\bm{A})^{-})^{-}\bm{p})^{1/2}\oplus\bm{q}^{-}\bm{A}\bm{g}$, and attained at $\bm{x}=\mu(\bm{q}^{-}\bm{A})^{-}$.
\end{theorem}

We now consider another problem with conditions that include a linear inequality defined by a matrix. Given vectors $\bm{p},\bm{q},\bm{g},\bm{h}\in\mathbb{X}^{n}$ and a matrix $\bm{B}\in\mathbb{X}^{n\times n}$, find regular vectors $\bm{x}\in\mathbb{X}^{n}$ that solve the problem
\begin{equation}
\begin{aligned}
&
\text{minimize}
&&
\bm{x}^{-}\bm{p}\oplus\bm{q}^{-}\bm{x},
%\\
&
\text{subject to}
&&
\bm{B}\bm{x}\oplus\bm{g}
\leq
\bm{x},
&&
\bm{x}
\leq
\bm{h}.
\end{aligned}
\label{P-xpqxBxxgxh}
\end{equation}

The solution of the problem follows an approach that is based on the general solution to linear inequalities proposed in \cite{Krivulin2006Solution,Krivulin2009Methods} and further refined in \cite{Krivulin2015Multidimensional}. We introduce an auxiliary parameter to represent the minimum value of the objective function. The problem is then reduced to the solving of a linear inequality with a matrix that depends on the parameter. We use the existence condition for solutions of the inequality to evaluate the parameter, and take all solutions of the inequality as a complete solution to the initial problem.

\begin{theorem}[\cite{Krivulin2014Complete}]
\label{T-xpqxBxxgxh}
Let $\bm{B}$ be a matrix such that $\mathop\mathrm{Tr}(\bm{B})\leq\mathbb{1}$, $\bm{p}$ be a nonzero vector, $\bm{q}$ and $\bm{h}$ regular vectors, and $\bm{g}$ a vector such that $\bm{h}^{-}\bm{B}^{\ast}\bm{g}\leq\bm{1}$. Then, the minimum in \eqref{P-xpqxBxxgxh} is equal to $\theta=(\bm{q}^{-}\bm{B}^{\ast}\bm{p})^{1/2}\oplus\bm{h}^{-}\bm{B}^{\ast}\bm{p}\oplus\bm{q}^{-}\bm{B}^{\ast}\bm{g}$, and all solutions are given by $\bm{x}=\bm{B}^{\ast}\bm{u}$, where $\bm{g}\oplus\theta^{-1}\bm{p}\leq\bm{u}\leq((\bm{h}^{-}\oplus\theta^{-1}\bm{q}^{-})\bm{B}^{\ast})^{-}$.
\end{theorem}

As a consequence, we obtain a complete solution to a problem that was partially solved in \cite{Krivulin2012New}. Consider a variant of problem \eqref{P-xpqxBxxgxh} in the form
\begin{equation}
\begin{aligned}
&
\text{minimize}
&&
\bm{x}^{-}\bm{p}\oplus\bm{q}^{-}\bm{x},
%\\
&
\text{subject to}
&&
\bm{B}\bm{x}
\leq
\bm{x}.
\end{aligned}
\label{P-xpqxBxx}
\end{equation}

\begin{corollary}[\cite{Krivulin2014Complete}]
\label{C-xpqxBxx}
Let $\bm{B}$ be a matrix with $\mathop\mathrm{Tr}(\bm{B})\leq\mathbb{1}$, and $\bm{p}$ be nonzero and $\bm{q}$ regular vectors. Then, the minimum in \eqref{P-xpqxBxx} is equal to $\theta=(\bm{q}^{-}\bm{B}^{\ast}\bm{p})^{1/2}$, and all solutions are given by $\bm{x}=\bm{B}^{\ast}\bm{u}$, where $\theta^{-1}\bm{p}\leq\bm{u}
\leq\theta(\bm{q}^{-}\bm{B}^{\ast})^{-}$.
\end{corollary}

\subsection{Problems with Span Seminorm}

Optimization problems, where the objective function takes the form of span seminorm, arose in the context of job scheduling \cite{Butkovic2009Onsome,Tam2010Optimizing}. Minimization of the span seminorm solves scheduling problems in just-in-time manufacturing. Maximization problems appear when the optimal schedule aims to spread the completion time of jobs over the maximum possible time interval. 

A solution to such problems without constraints is based on the evaluation of lower or upper bounds for the objective function. If a problem has linear equation or inequality constraints, we first obtain a general solution to the equation or inequality, and then substitute it into the objective function to reduce to an unconstrained problem with known solution.

\subsubsection{Minimization Problems}

Consider an unconstrained problem that is an extension of the minimization problem at \eqref{P-minmax1AxAx1}. Given matrices $\bm{A},\bm{B}\in\mathbb{X}^{m\times n}$ and vectors 
$\bm{p},\bm{q}\in\mathbb{X}^{m}$, the problem is to find regular vectors $\bm{x}\in\mathbb{X}^{n}$ that
\begin{equation}
\begin{aligned}
&
\text{minimize}
&&
\bm{q}^{-}\bm{B}\bm{x}(\bm{A}\bm{x})^{-}\bm{p}.
\end{aligned}
\label{P-minqBxAxp}
\end{equation}

\begin{theorem}[\cite{Krivulin2013Explicit}]
\label{T-minqBxAxp}
Let $\bm{A}$ be row-regular and $\bm{B}$ column-regular matrices, and $\bm{p}$ be nonzero and $\bm{q}$ regular vectors. Then, the minimum in \eqref{P-minqBxAxp} is equal to $\Delta=(\bm{A}(\bm{q}^{-}\bm{B})^{-})^{-}\bm{p}$, and attained at $\bm{x}=\alpha(\bm{q}^{-}\bm{B})^{-}$ for all $\alpha>\mathbb{0}$.
\end{theorem}

We now examine some special cases of problem \eqref{P-minqBxAxp}. First, assume that $\bm{B}=\bm{A}=\bm{I}$ and $\bm{p}=\bm{q}=\bm{1}$, and consider the problem
\begin{equation*}
\begin{aligned}
&
\text{minimize}
&&
\bm{1}^{T}\bm{x}\bm{x}^{-}\bm{1}.
\end{aligned}
\end{equation*}

A direct application of Theorem~\ref{T-minqBxAxp} shows that the problem has the minimum $\Delta=\mathbb{1}$, which is attained at any vector $\bm{x}=\alpha\bm{1}$ for all $\alpha>\mathbb{0}$. 

Under the assumptions that $\bm{A}=\bm{B}$ and $\bm{p}=\bm{q}=\bm{1}$, we arrive at the minimization problem at \eqref{P-minmax1AxAx1}. Application of Theorem~\ref{T-minqBxAxp} provides a new solution to this problem in a compact vector form.
\begin{corollary}
\label{C-min1AxAx1}
Let $\bm{A}$ be a regular matrix. Then, the minimum in \eqref{P-minmax1AxAx1} is equal to $\Delta=(\bm{A}(\bm{1}^{T}\bm{A})^{-})^{-}\bm{1}$, and attained at $\bm{x}=\alpha(\bm{1}^{T}\bm{A})^{-}$ for all $\alpha>\mathbb{0}$.
\end{corollary}

Furthermore, we consider the following constrained problem: given matrices $\bm{C},\bm{D}\in\mathbb{X}^{n\times n}$, find regular vectors $\bm{x}\in\mathbb{X}^{n}$ that
\begin{equation}
\begin{aligned}
&
\text{minimize}
&&
\bm{1}^{T}\bm{y}\bm{y}^{-}\bm{1},
%\\
&
\text{subject to}
&&
\bm{C}\bm{x}
=
\bm{y},
&&
\bm{D}\bm{x}
\leq
\bm{x}.
\end{aligned}
\label{P-1yy1CxyDxxI}
\end{equation}

\begin{theorem}[\cite{Krivulin2013Explicit}]
\label{T-1yy1CxyDxxI}
Let $\bm{C}$ be a regular matrix and $\bm{D}$ be a matrix with $\mathop\mathrm{Tr}(\bm{D})\leq\mathbb{1}$. Then, the minimum in \eqref{P-1yy1CxyDxxI} is equal to $\Delta=(\bm{C}\bm{D}^{\ast}(\bm{1}^{T}\bm{C}\bm{D}^{\ast})^{-})^{-}\bm{1}$, and attained at $\bm{x}=\alpha\bm{D}^{\ast}(\bm{1}^{T}\bm{C}\bm{D}^{\ast})^{-}$ for all $\alpha>\mathbb{0}$.
\end{theorem}

\subsubsection{Maximization Problems}

Suppose that, given matrices $\bm{A}\in\mathbb{X}^{m\times n}$ and $\bm{B}\in\mathbb{X}^{l\times n}$, and vectors $\bm{p}\in\mathbb{X}^{m}$, $\bm{q}\in\mathbb{X}^{l}$, we find regular vectors $\bm{x}\in\mathbb{X}^{n}$ that
\begin{equation}
\begin{aligned}
&
\text{maximize}
&&
\bm{q}^{-}\bm{B}\bm{x}(\bm{A}\bm{x})^{-}\bm{p}.
\end{aligned}
\label{P-maxqBxAxp}
\end{equation}

A complete solution to \eqref{P-maxqBxAxp} under fairly general conditions is as follows.
\begin{theorem}[\cite{Krivulin2016Maximization}]
\label{T-maxqBxAxp}
Let $\bm{A}$ be a matrix with regular columns, $\bm{B}$ be a column-regular matrix, $\bm{p}$ and $\bm{q}$ be regular vectors. Then, the minimum in \eqref{P-maxqBxAxp} is equal to $\Delta=\bm{q}^{-}\bm{B}\bm{A}^{-}\bm{p}$, and any solution $\bm{x}=(x_{i})$ has the components given by $x_{k}=\alpha\bm{a}_{k}^{-}\bm{p}$ and $x_{j}\leq\alpha a_{sj}^{-1}p_{s}$ for all $j\ne k$, where $\alpha>\mathbb{0}$, and the indices $k$ and $s$ are defined by the conditions
$$
k
=
\arg\max_{1\leq i\leq n}\bm{q}^{-}\bm{b}_{i}\bm{a}_{i}^{-}\bm{p},
\qquad
s
=
\arg\max_{1\leq i\leq m}a_{ik}^{-1}p_{i}.
$$
\end{theorem}

Now assume that $\bm{p}=\bm{q}=\bm{1}$. Then, we can write
$$
\bm{1}^{T}\bm{B}\bm{x}(\bm{A}\bm{x})^{-}\bm{1}
=
\|\bm{B}\bm{x}\|\|(\bm{A}\bm{x})^{-}\|,
\qquad
\bm{1}^{T}\bm{B}\bm{A}^{-}\bm{1}
=
\|\bm{B}\bm{A}^{-}\|.
$$

Under this assumption, problem \eqref{P-maxqBxAxp} takes the form
\begin{equation}
\begin{aligned}
&
\text{maximize}
&&
\|\bm{B}\bm{x}\|\|(\bm{A}\bm{x})^{-}\|.
\end{aligned}
\label{P-max1BxAx1}
\end{equation}

By applying Theorem~\ref{T-maxqBxAxp}, we obtain the following result.
\begin{corollary}[\cite{Krivulin2016Maximization}]
\label{C-max1BxAx1}
Let $\bm{A}$ be a matrix with regular columns, and $\bm{B}$ a column-regular matrix. Then, the minimum in \eqref{P-max1BxAx1} is equal to $\Delta=\|\bm{B}\bm{A}^{-}\|$, and any solution $\bm{x}=(x_{i})$ has the components given by $x_{k}=\alpha\|\bm{a}_{k}^{-}\|$ and $x_{j}\leq\alpha a_{sj}^{-1}$ for all $j\ne k$, where $\alpha>\mathbb{0}$, and the indices $k$ and $s$ are defined by the conditions
$$
k
=
\arg\max_{1\leq i\leq n}\|\bm{b}_{i}\|\|\bm{a}_{i}^{-}\|,
\qquad
s
=
\arg\max_{1\leq i\leq m}a_{ik}^{-1}.
$$
\end{corollary}

We now turn to the solution of a constrained problem: given a matrix $\bm{C}\in\mathbb{X}^{n\times n}$, we have to solve the problem 
\begin{equation}
\begin{aligned}
&
\text{maximize}
&&
\bm{q}^{-}\bm{B}\bm{x}(\bm{A}\bm{x})^{-}\bm{p},
&
\text{subject to}
&&
\bm{C}\bm{x}
\leq
\bm{x}.
\end{aligned}
\label{P-maxqBxAxpCxlex}
\end{equation}

By Theorem~\ref{T-Axbx}, the inequality constraint has regular solutions if and only if $\mathop\mathrm{Tr}(\bm{C})\leq\mathbb{1}$. Under this condition, all solutions to the inequality are given by $\bm{x}=\bm{C}^{\ast}\bm{u}$, where $\bm{u}$ is any regular vector. Substitution of the solutions into the objective function reduces problem \eqref{P-maxqBxAxpCxlex} to the unconstrained problem
\begin{equation*}
\begin{aligned}
&
\text{maximize}
&&
\bm{q}^{-}\bm{B}\bm{C}^{\ast}\bm{u}(\bm{A}\bm{C}^{\ast}\bm{u})^{-}\bm{p}.
\end{aligned}
\label{P-maxqBCuACup-ast}
\end{equation*}

This problem has the form of \eqref{P-maxqBxAxp}, and thus is solved by Theorem~\ref{T-maxqBxAxp}.

\subsection{Problems with Evaluation of Spectral Radius}

We return to problem \eqref{P-minxAx} with the minimum value given by the spectral radius. All solutions to this problem are obtained in a closed form in \cite{Krivulin2015Multidimensional} as a consequence of the solution to a more general optimization problem. A direct complete solution to problem \eqref{P-minxAx} is derived in \cite{Krivulin2015Extremal} in the following form.
\begin{lemma}\label{L-minxAx}
Let $\bm{A}$ be a matrix with spectral radius $\lambda>\mathbb{0}$. Then, the minimum in \eqref{P-minxAx} is equal to $\lambda$, and all solutions are given by $\bm{x}=(\lambda^{-1}\bm{A})^{\ast}\bm{u}$, $\bm{u}\in\mathbb{X}^{n}$.
\end{lemma}

The proof of the statement involves the derivation of a sharp lower bound on the objective function. An equation is written as an equality between the function and the bound to specify all solutions of the problem. We reduce the equation to an inequality and then take all solutions of the inequality as a complete solution of the optimization problem.

Further extensions of problem \eqref{P-minxAx} with more general form of the objective function and additional constraints are examined in \cite{Krivulin2012Acomplete,Krivulin2014Constrained}. Below, we offer complete, direct solutions to certain new problems that extend \eqref{P-minxAx}.

\subsubsection{An Unconstrained Problem}

Given a matrix $\bm{A}\in\mathbb{X}^{n\times n}$, vectors $\bm{p},\bm{q}\in\mathbb{X}^{n}$, and a scalar $r\in\mathbb{X}$, the problem is to obtain regular vectors $\bm{x}\in\mathbb{X}^{n}$ that
\begin{equation}
\begin{aligned}
&
\text{minimize}
&&
\bm{x}^{-}\bm{A}\bm{x}\oplus\bm{x}^{-}\bm{p}\oplus\bm{q}^{-}\bm{x}\oplus r.
\end{aligned}
\label{P-xAxxpqxc}
\end{equation}

The problem is completely solved by the following result.
\begin{theorem}[\cite{Krivulin2015Extremal}]
\label{T-xAxxpqxc}
Let $\bm{A}$ be a matrix with spectral radius $\lambda>\mathbb{0}$, $\bm{q}$ be a regular vector. Then, the minimum value in problem \eqref{P-xAxxpqxc} is equal to
\begin{equation*}
\mu
=
\lambda
\oplus
\bigoplus_{m=1}^{n}
(\bm{q}^{-}\bm{A}^{m-1}\bm{p})^{1/(m+1)}
\oplus
r,
\end{equation*}
and all solutions are given by $\bm{x}=(\mu^{-1}\bm{A})^{\ast}\bm{u}$, $\mu^{-1}\bm{p}\leq\bm{u}\leq\mu(\bm{q}^{-}(\mu^{-1}\bm{A})^{\ast})^{-}$.
\end{theorem}

\subsubsection{Problems with Constraints}

Suppose that, given matrices $\bm{A},\bm{B}\in\mathbb{X}^{n\times n}$, $\bm{C}\in\mathbb{X}^{m\times n}$, and vectors $\bm{g}\in\mathbb{X}^{n}$, $\bm{h}\in\mathbb{X}^{m}$, we find regular vectors $\bm{x}\in\mathbb{X}^{n}$ that
\begin{equation}
\begin{aligned}
&
\text{minimize}
&&
\bm{x}^{-}\bm{A}\bm{x},
%\\
&
\text{subject to}
&&
\bm{B}\bm{x}\oplus\bm{g}
\leq
\bm{x},
&&
\bm{C}
\bm{x}
\leq
\bm{h}.
\end{aligned}
\label{P-xAxBxgxCxh}
\end{equation}

\begin{theorem}[\cite{Krivulin2014Constrained}]
\label{T-xAxBxgxCxh}
Let $\bm{A}$ be a matrix with spectral radius $\lambda>\mathbb{0}$, $\bm{B}$ be a matrix such that $\mathop\mathrm{Tr}(\bm{B})\leq\mathbb{1}$, $\bm{C}$ be a column-regular matrix, and $\bm{h}$ be a regular vector such that $\bm{h}^{-}\bm{C}\bm{B}^{\ast}\bm{g}\leq\bm{1}$. Then, the minimum in \eqref{P-xAxBxgxCxh} is equal to
\begin{equation*}
\theta
=
\bigoplus_{k=1}^{n}\mathop{\bigoplus\hspace{0.0em}}_{0\leq i_{0}+
i_{1}+\cdots+i_{k}\leq n-k}\mathop\mathrm{tr}\nolimits^{1/k}(\bm{B}^{i_{0}}
(\bm{A}\bm{B}^{i_{1}}\cdots\bm{A}\bm{B}^{i_{k}})(\bm{I}\oplus\bm{g}\bm{h}^{-}\bm{C})),
%\label{E-theta}
\end{equation*}
and all solutions are given by $\bm{x}=(\theta^{-1}\bm{A}\oplus\bm{B})^{\ast}\bm{u}$, where the vector $\bm{u}$ satisfies the condition $\bm{g}\leq\bm{u}\leq(\bm{h}^{-}\bm{C}(\theta^{-1}\bm{A}\oplus\bm{B})^{\ast})^{-}$.
\end{theorem}

Consider the special case with $\bm{C}=\bm{0}$. Problem \eqref{P-xAxBxgxCxh} takes the form 
\begin{equation}
\begin{aligned}
&
\text{minimize}
&&
\bm{x}^{-}\bm{A}\bm{x},
%\\
&
\text{subject to}
&&
\bm{B}\bm{x}\oplus\bm{g}
\leq
\bm{x}.
\end{aligned}
\label{P-xAxBxgx}
\end{equation}

\begin{corollary}[\cite{Krivulin2015Multidimensional,Krivulin2014Constrained}]
\label{C-xAxBxgx}
Let $\bm{A}$ be a matrix with spectral radius $\lambda>\mathbb{0}$, and $\bm{B}$ be a matrix such that $\mathop\mathrm{Tr}(\bm{B})\leq\mathbb{1}$. Then, the minimum in \eqref{P-xAxBxgx} is equal to
$$
\theta
=
\lambda
\oplus
\bigoplus_{k=1}^{n-1}\mathop{\bigoplus\hspace{1.2em}}_{1\leq i_{1}+\cdots+i_{k}\leq n-k}
\mathop\mathrm{tr}\nolimits^{1/k}(\bm{A}\bm{B}^{i_{1}}\cdots\bm{A}\bm{B}^{i_{k}}),
$$
and all solutions are given by $\bm{x}=(\theta^{-1}\bm{A}\oplus\bm{B})^{\ast}\bm{u}$, where $\bm{u}\geq\bm{g}$.
\end{corollary}

Under the conditions $\bm{B}=\bm{0}$ and $\bm{C}=\bm{I}$ problem \eqref{P-xAxBxgxCxh} becomes
\begin{equation}
\begin{aligned}
&
\text{minimize}
&&
\bm{x}^{-}\bm{A}\bm{x},
%\\
&
\text{subject to}
&&
\bm{g}
\leq
\bm{x}
\leq
\bm{h}.
\end{aligned}
\label{P-xAxgxh}
\end{equation}

\begin{corollary}[\cite{Krivulin2014Constrained}]
\label{C-xAxgxh}
Let $\bm{A}$ be a matrix with spectral radius $\lambda>\mathbb{0}$, and $\bm{h}$ be a regular vector such that $\bm{h}^{-}\bm{g}\leq\bm{1}$. Then, the minimum in \eqref{P-xAxgxh} is equal to
\begin{equation*}
\theta
=
\lambda\oplus\bigoplus_{k=1}^{n}(\bm{h}^{-}\bm{A}^{k}\bm{g})^{1/k},
\end{equation*}
and all solutions are given by $\bm{x}=(\theta^{-1}\bm{A})^{\ast}\bm{u}$, where $\bm{g}\leq\bm{u}\leq(\bm{h}^{-}(\theta^{-1}\bm{A})^{\ast})^{-}$.
\end{corollary}

The next problem combines the special case of the objective function in problem \eqref{P-xAxxpqxc} with the constraints in \eqref{P-xAxBxgx}. Given matrices $\bm{A},\bm{B}\in\mathbb{X}^{n\times n}$ and vectors $\bm{p},\bm{g}\in\mathbb{X}^{n}$, the problem is to obtain regular vectors $\bm{x}\in\mathbb{X}^{n}$ that
\begin{equation}
\begin{aligned}
&
\text{minimize}
&&
\bm{x}^{-}\bm{A}\bm{x}\oplus\bm{x}^{-}\bm{p},
%\\
&
\text{subject to}
&&
\bm{B}\bm{x}\oplus\bm{g}
\leq
\bm{x}.
\end{aligned}
\label{P-xAxxpBxgx}
\end{equation}

\begin{theorem}[\cite{Krivulin2015Extremal}]
\label{T-xAxxpBxgx}
Let $\bm{A}$ be a matrix with spectral radius $\lambda>\mathbb{0}$, and $\bm{B}$ be a matrix with $\mathop\mathrm{Tr}(\bm{B})\leq\bm{1}$. Then, the minimum in \eqref{P-xAxxpBxgx} is equal to
\begin{equation*}
\theta
=
\lambda
\oplus
\bigoplus_{k=1}^{n-1}\mathop{\bigoplus\hspace{1.2em}}_{1\leq i_{1}+
\cdots+i_{k}\leq n-k}\mathop\mathrm{tr}\nolimits^{1/k}(\bm{A}\bm{B}^{i_{1}}\cdots\bm{A}\bm{B}^{i_{k}}),
%\label{E-theta}
\end{equation*}
and all solutions are given by $\bm{x}=(\theta^{-1}\bm{A}\oplus\bm{B})^{\ast}\bm{u}$, where $\bm{u}\geq\theta^{-1}\bm{p}\oplus\bm{g}$.
\end{theorem}

Another extended problem that has an objective function like that in \eqref{P-xAxxpqxc} and boundary constraints as in \eqref{P-xAxgxh} is examined in the next section.

\section{New Constrained Optimization Problem}
\label{S-NCOP}

This section includes a direct, complete solution to a new constrained optimization problem. We follow the approach developed in \cite{Krivulin2015Multidimensional,Krivulin2014Complete,Krivulin2014Constrained} to introduce an additional variable, which represents the minimum value of the objective function, and then to reduce the problem to solving a parametrized inequality. 

Suppose that, given a matrix $\bm{A}\in\mathbb{X}^{n\times n}$, vectors $\bm{p},\bm{q},\bm{g},\bm{h}\in\mathbb{X}^{n}$, and a scalar $r\in\mathbb{X}$, we need to find regular vectors $\bm{x}\in\mathbb{X}^{n}$ that solve the problem
\begin{equation}
\begin{aligned}
&
\text{minimize}
&&
\bm{x}^{-}\bm{A}\bm{x}\oplus\bm{x}^{-}\bm{p}\oplus\bm{q}^{-}\bm{x}\oplus r,
%\\
&
\text{subject to}
&&
\bm{g}
\leq
\bm{x}
\leq
\bm{h}.
\end{aligned}
\label{P-xAxxpqxr-gxh}
\end{equation}

\begin{theorem}
Let $\bm{A}$ be a matrix with spectral radius $\lambda>\mathbb{0}$, $\bm{q}$ and $\bm{h}$ be regular vectors, and $\bm{h}^{-}\bm{g}\leq\mathbb{1}$. Then, the minimum in \eqref{P-xAxxpqxr-gxh} is equal to
\begin{multline*}
\mu
=
\lambda
\oplus
\bigoplus_{m=0}^{n-1}
(\bm{q}^{-}\bm{A}^{m}\bm{p})^{1/(m+2)}
\oplus
\bigoplus_{m=0}^{n-1}
(\bm{q}^{-}\bm{A}^{m}\bm{g}\oplus\bm{h}^{-}\bm{A}^{m}\bm{p})^{1/(m+1)}
\\
\oplus
\bigoplus_{m=1}^{n-1}
(\bm{h}^{-}\bm{A}^{m}\bm{g})^{1/m}
\oplus
r.
\end{multline*}
and all regular solutions are given by $\bm{x}=(\mu^{-1}\bm{A})^{\ast}\bm{u}$, where the vector $\bm{u}$ satisfies the condition $\mu^{-1}\bm{p}\oplus\bm{g}\leq\bm{u}\leq((\mu^{-1}\bm{q}^{-}\oplus\bm{h}^{-})(\mu^{-1}\bm{A})^{\ast})^{-}$.
\end{theorem}

\begin{proof}
Denote the minimum value of the objective function over all regular $\bm{x}$ by $\mu$. Then, all regular solutions to problem \eqref{P-xAxxpqxr-gxh} are given by the system
\begin{equation}
\bm{x}^{-}\bm{A}\bm{x}\oplus\bm{x}^{-}\bm{p}\oplus\bm{q}^{-}\bm{x}\oplus r
\leq
\mu,
\qquad
\bm{g}
\leq
\bm{x}
\leq
\bm{h}.
\label{I-xAxxpqxrmu-gxh}
\end{equation}

The first inequality at \eqref{I-xAxxpqxrmu-gxh} is equivalent to the system of inequalities
$$
\bm{x}^{-}\bm{A}\bm{x}
\leq
\mu,
\qquad
\bm{x}^{-}\bm{p}
\leq
\mu,
\qquad
\bm{q}^{-}\bm{x}
\leq
\mu,
\qquad
r
\leq
\mu.
$$

Note that, by Lemma~\ref{L-minxAx}, we have $\mu\geq\bm{x}^{-}\bm{A}\bm{x}\geq\lambda$. Furthermore, after multiplication of the second and third inequalities, and application of properties of conjugate transposition, we obtain $\bm{q}^{-}\bm{p}\leq\bm{q}^{-}\bm{x}\bm{x}^{-}\bm{p}\leq\mu^{2}$, and thus $\mu\geq(\bm{q}^{-}\bm{p})^{1/2}$. With the forth inequality, we derive the lower bound for $\mu$
\begin{equation}
\mu
\geq
\lambda\oplus(\bm{q}^{-}\bm{p})^{1/2}\oplus r.
\label{I-mulambdaqpr}
\end{equation}

Furthermore, applications of Lemma~\ref{I-Axd} to the first three inequalities, followed by multiplication by $\mu^{-1}$, gives
$$
\mu^{-1}\bm{A}\bm{x}
\leq
\bm{x},
\qquad
\mu^{-1}\bm{p}
\leq
\bm{x},
\qquad
\bm{x}
\leq
\mu\bm{q}.
$$

Finally, we combine the first two inequalities with the left boundary constraint $\bm{g}\leq\bm{x}$, and then the third inequality with the right boundary $\bm{x}\leq\bm{h}$ to rewrite the system at \eqref{I-xAxxpqxrmu-gxh} as the double inequality
\begin{equation*}
\mu^{-1}\bm{A}\bm{x}\oplus\mu^{-1}\bm{p}\oplus\bm{g}
\leq
\bm{x}
\leq
(\mu^{-1}\bm{q}^{-}\oplus\bm{h}^{-})^{-}.
%\label{I-muAxmupgx-xmuqh}
\end{equation*}

To solve the left inequality, we need to apply Theorem~\ref{T-Axbx}. First, we verify that $\mathop\mathrm{Tr}(\mu^{-1}\bm{A})\leq\mathop\mathrm{Tr}(\lambda^{-1}\bm{A})=\lambda^{-1}\mathop\mathrm{tr}\bm{A}\oplus\cdots\oplus\lambda^{-n}\mathop\mathrm{tr}\bm{A}^{n}\leq\mathbb{1}$. It follows from the theorem that the left inequality has regular solutions, all given by $\bm{x}=(\mu^{-1}\bm{A})^{\ast}\bm{u}$, where $\bm{u}$ is a regular vector such that $\bm{u}\geq\mu^{-1}\bm{p}\oplus\bm{g}$.

Substitution of the solution into the right inequality leads to the inequality $(\mu^{-1}\bm{A})^{\ast}\bm{u}\leq(\mu^{-1}\bm{q}^{-}\oplus\bm{h}^{-})^{-}$, which is completely solved with respect to $\bm{u}$ by Lemma~\ref{I-Axd} in the form $\bm{u}\leq((\mu^{-1}\bm{q}^{-}\oplus\bm{h}^{-})(\mu^{-1}\bm{A})^{\ast})^{-}$.

By coupling both left and right boundary conditions, we write the double inequality $\mu^{-1}\bm{p}\oplus\bm{g}\leq\bm{u}\leq((\mu^{-1}\bm{q}^{-}\oplus\bm{h}^{-})(\mu^{-1}\bm{A})^{\ast})^{-}$, which determines a nonempty set if and only if $\mu^{-1}\bm{p}\oplus\bm{g}\leq((\mu^{-1}\bm{q}^{-}\oplus\bm{h}^{-})(\mu^{-1}\bm{A})^{\ast})^{-}$.

Using properties of conjugate transposition, it is easy to verify that the last inequality is equivalent to that in the form
$$
(\mu^{-1}\bm{q}^{-}\oplus\bm{h}^{-})(\mu^{-1}\bm{A})^{\ast}(\mu^{-1}\bm{p}\oplus\bm{g})
\leq
\mathbb{1}.
$$

By simple algebra, we reduce the last inequality to the system of inequalities
\begin{gather*}
\mu^{-2}\bm{q}^{-}(\mu^{-1}\bm{A})^{\ast}\bm{p}
\leq
\mathbb{1},
\qquad
\mu^{-1}(\bm{q}^{-}(\mu^{-1}\bm{A})^{\ast}\bm{g}
\oplus
\bm{h}^{-}(\mu^{-1}\bm{A})^{\ast}\bm{p})
\leq
\mathbb{1},
\\
\bm{h}^{-}(\mu^{-1}\bm{A})^{\ast}\bm{g}
\leq
\mathbb{1}.
\end{gather*}

The replacement of the asterate by its expanded form yields
\begin{gather*}
\bigoplus_{m=0}^{n-1}\mu^{-m-2}\bm{q}^{-}\bm{A}^{m}\bm{p}
\leq
\mathbb{1},
\qquad
\bigoplus_{m=0}^{n-1}\mu^{-m-1}(\bm{q}^{-}\bm{A}^{m}\bm{g}\oplus\bm{h}^{-}\bm{A}^{m}\bm{p})
\leq
\mathbb{1},
\\
\bm{h}^{-}\bm{g}
\oplus
\bigoplus_{m=1}^{n-1}\mu^{-m}\bm{h}^{-}\bm{A}^{m}\bm{g}
\leq
\mathbb{1}.
\end{gather*}

Considering that $\bm{h}^{-}\bm{g}\leq\mathbb{1}$ by the conditions of the theorem, the last inequalities are equivalent to the system
\begin{gather*}
\mu^{-m-2}\bm{q}^{-}\bm{A}^{m}\bm{p}
\leq
\mathbb{1},
\quad
\mu^{-m-1}(\bm{q}^{-}\bm{A}^{m}\bm{g}\oplus\bm{h}^{-}\bm{A}^{m}\bm{p})
\leq
\mathbb{1},
\quad
m=0,\ldots,n-1;
\\
\mu^{-m}\bm{h}^{-}\bm{A}^{m}\bm{g}
\leq
\mathbb{1},
\quad
m=1,\ldots,n-1;
\end{gather*}
which can be solved with respect to $\mu$ in the form
\begin{gather*}
\mu
\geq
(\bm{q}^{-}\bm{A}^{m}\bm{p})^{1/(m+2)},
\quad
\mu
\geq
(\bm{q}^{-}\bm{A}^{m}\bm{g}\oplus\bm{h}^{-}\bm{A}^{m}\bm{p})^{1/(m+1)},
\quad
m=0,\ldots,n-1;
\\
\mu
\geq
(\bm{h}^{-}\bm{A}^{m}\bm{g})^{1/m},
\quad
m=1,\ldots,n-1.
\end{gather*}

Furthermore, we combine these solution into one inequality
$$
\mu
\geq
\bigoplus_{m=0}^{n-1}
(\bm{q}^{-}\bm{A}^{m}\bm{p})^{1/(m+2)}
\oplus
\bigoplus_{m=0}^{n-1}
(\bm{q}^{-}\bm{A}^{m}\bm{g}\oplus\bm{h}^{-}\bm{A}^{m}\bm{p})^{1/(m+1)}
\oplus
\bigoplus_{m=1}^{n-1}
(\bm{h}^{-}\bm{A}^{m}\bm{g})^{1/m}.
$$

We add the bound at \eqref{I-mulambdaqpr}, and then replace the obtained inequality by equality to provide the minimum value of $\mu$, and thus completes the proof. 
\end{proof}

\bibliographystyle{abbrvurl}

\bibliography{Algebraic_solutions_of_tropical_optimization_problems}

\end{document}